\documentclass[10pt]{extarticle}

\usepackage[english]{babel}
\usepackage{graphicx}
\usepackage{framed}
\usepackage[normalem]{ulem}
\usepackage{indentfirst,ragged2e}
\usepackage{amsmath,amsthm,amssymb,amsfonts,wasysym,verbatim,bbm}
\usepackage[italicdiff]{physics}
\usepackage[T1]{fontenc}
\usepackage{lmodern,mathrsfs}
\usepackage{mathdots} 
\usepackage{wrapfig}
\usepackage[inline,shortlabels]{enumitem}
\setlist{topsep=2pt,itemsep=2pt,parsep=0pt,partopsep=0pt}
\usepackage[table,dvipsnames]{xcolor}
\usepackage[utf8]{inputenc}
\usepackage[a4paper,top=0.7in,bottom=0.7in,left=0.5in,right=0.5in]{geometry}
\usepackage{multicol}
\usepackage[most]{tcolorbox}
\usepackage{nicematrix} 
\usepackage{tikz,tikz-3dplot,tikz-cd,tkz-tab,tkz-euclide,tikzsymbols,pgf,pgfplots}
\pgfplotsset{compat=1.3}
\usepgfplotslibrary{fillbetween}
\usepackage{subfiles}
\graphicspath{{images/}{../images/}} 
\usepackage[backend=bibtex,style=numeric]{biblatex}
\usepackage{hyperref}
\usepackage[nameinlink]{cleveref}

\newcommand{\hide}[1]{} 

\usepackage[showisoZ=false]{datetime2}

\DTMnewdatestyle{newdate}{

}

\newcommand*{\utcpm}[1]{
\ifboolexpe{test{\ifnumequal{#1}{0}}}{}{\ifnum#1<0\else+\fi{#1}}}
\DTMnewzonestyle{newzone}{
    
}
\AtBeginDocument{
    \DTMsetdatestyle{newdate}
    \DTMsetzonestyle{newzone}
}

\usepackage{fancyhdr}
\fancypagestyle{plain}{
\fancyhf{}

\fancyfoot[C]{\sffamily\thepage}}

\newcommand{\N}{\mathbb{N}}
\newcommand{\Q}{\mathbb{Q}}
\newcommand{\R}{\mathbb{R}}

\newcommand{\Ls}{\mathcal{L}}

\newcommand{\Bf}{\mathfrak{B}}
\newcommand{\Mf}{\mathfrak{M}}

\newcommand{\limn}{\lim_{n\to\infty}}

\newcommand{\sumk}[1][1]{\sum_{k=#1}^\infty}

\newcommand{\exc}{\backslash}
\newcommand{\sub}{\subseteq}
\newcommand{\sups}{\supseteq}

\newcommand{\cupp}{\bigcup}

\newcommand{\cuppk}[1][1]{\bigcup_{k=#1}^\infty}

\newcommand{\cappn}[1][1]{\bigcap_{n=#1}^\infty}
\newcommand{\cuppn}[1][1]{\bigcup_{n=#1}^\infty}

\newcommand{\dx}{\,dx}

\newcommand*{\term}[1]{\textcolor{red!20!black}{\textbf{#1}}}
\newcommand*{\setbuild}[2]{\left\{#1\,\middle|\,#2\right\}}

\newtheoremstyle{mystyle}{}{}{}{}{\sffamily\bfseries}{.}{ }{}
\newtheoremstyle{cstyle}{}{}{}{}{\sffamily\bfseries}{.}{ }{\thmnote{#3}}
\makeatletter
\renewenvironment{proof}[1][\proofname] {\par\pushQED{\qed}{\normalfont\sffamily\bfseries\topsep6\p@\@plus6\p@\relax #1\@addpunct{.} }}{\popQED\endtrivlist\@endpefalse}
\makeatother
\theoremstyle{mystyle}{\newtheorem{definition}{Definition}[section]}
\theoremstyle{mystyle}{\newtheorem{proposition}[definition]{Proposition}}
\theoremstyle{mystyle}{\newtheorem{theorem}[definition]{Theorem}}
\theoremstyle{mystyle}{\newtheorem{lemma}[definition]{Lemma}}
\theoremstyle{mystyle}{\newtheorem{corollary}[definition]{Corollary}}
\theoremstyle{mystyle}{\newtheorem*{remark}{Remark}}
\theoremstyle{mystyle}{\newtheorem*{remarks}{Remarks}}
\theoremstyle{mystyle}{}
\theoremstyle{mystyle}{}
\theoremstyle{definition}{}
\theoremstyle{cstyle}{}
\newtheoremstyle{warn}{}{}{}{}{\normalfont}{}{ }{}
\theoremstyle{warn}

\newcommand{\warningsign}[1]{\tikz[scale=#1,every node/.style={transform shape}]{
\draw[-,line width={#1*0.8mm},red,fill=yellow,rounded corners={#1*2.5mm}] (0,0)--(1,{-sqrt(3)})--(-1,{-sqrt(3)})--cycle;
\node at (0,-1) {\fontsize{48}{60}\selectfont\bfseries!};}}

\makeatletter
\def\ifemptyarg#1{%
  \if\relax\detokenize{#1}\relax 
    \expandafter\@firstoftwo
  \else
    \expandafter\@secondoftwo
  \fi}
\makeatother

\tcolorboxenvironment{definition}{boxrule=0pt,boxsep=0pt,colback={red!10},left=8pt,right=8pt,enhanced jigsaw, borderline west={2pt}{0pt}{red},sharp corners,before skip=10pt,after skip=10pt,breakable}
\tcolorboxenvironment{proposition}{boxrule=0pt,boxsep=0pt,colback={Orange!10},left=8pt,right=8pt,enhanced jigsaw, borderline west={2pt}{0pt}{Orange},sharp corners,before skip=10pt,after skip=10pt,breakable}
\tcolorboxenvironment{theorem}{boxrule=0pt,boxsep=0pt,colback={blue!10},left=8pt,right=8pt,enhanced jigsaw, borderline west={2pt}{0pt}{blue},sharp corners,before skip=10pt,after skip=10pt,breakable}
\tcolorboxenvironment{lemma}{boxrule=0pt,boxsep=0pt,colback={Cyan!10},left=8pt,right=8pt,enhanced jigsaw, borderline west={2pt}{0pt}{Cyan},sharp corners,before skip=10pt,after skip=10pt,breakable}
\tcolorboxenvironment{corollary}{boxrule=0pt,boxsep=0pt,colback={violet!10},left=8pt,right=8pt,enhanced jigsaw, borderline west={2pt}{0pt}{violet},sharp corners,before skip=10pt,after skip=10pt,breakable}
\tcolorboxenvironment{proof}{boxrule=0pt,boxsep=0pt,blanker,borderline west={2pt}{0pt}{CadetBlue!80!white},left=8pt,right=8pt,sharp corners,before skip=10pt,after skip=10pt,breakable}
\tcolorboxenvironment{remark}{boxrule=0pt,boxsep=0pt,blanker,borderline west={2pt}{0pt}{Green},left=8pt,right=8pt,before skip=10pt,after skip=10pt,breakable}
\tcolorboxenvironment{remarks}{boxrule=0pt,boxsep=0pt,blanker,borderline west={2pt}{0pt}{Green},left=8pt,right=8pt,before skip=10pt,after skip=10pt,breakable}
\tcolorboxenvironment{example}{boxrule=0pt,boxsep=0pt,blanker,borderline west={2pt}{0pt}{Black},left=8pt,right=8pt,sharp corners,before skip=10pt,after skip=10pt,breakable}
\tcolorboxenvironment{examples}{boxrule=0pt,boxsep=0pt,blanker,borderline west={2pt}{0pt}{Black},left=8pt,right=8pt,sharp corners,before skip=10pt,after skip=10pt,breakable}
\tcolorboxenvironment{cthm}{boxrule=0pt,boxsep=0pt,colback={gray!10},left=8pt,right=8pt,enhanced jigsaw, borderline west={2pt}{0pt}{gray},sharp corners,before skip=10pt,after skip=10pt,breakable}

\AddToHook{env/proposition/begin}{\crefalias{definition}{proposition}}
\AddToHook{env/lemma/begin}{\crefalias{definition}{lemma}}
\AddToHook{env/theorem/begin}{\crefalias{definition}{theorem}}
\AddToHook{env/corollary/begin}{\crefalias{definition}{corollary}}

\newenvironment{talign*}{\csname align*\endcsname}{\endalign}

\usepackage[explicit]{titlesec}
\titleformat{\section}{\normalfont\sffamily\Large\bfseries}{\thesection}{12pt}{#1}
\titleformat{\subsection}{\normalfont\sffamily\Large\bfseries}{\thesubsection}{12pt}{#1}
\titleformat{\subsubsection}{\normalfont\sffamily\large\bfseries}{\thesubsection}{8pt}{#1}

\titlespacing*{\section}{0pt}{5pt}{5pt}
\titlespacing*{\subsection}{0pt}{5pt}{5pt}
\titlespacing*{\subsubsection}{0pt}{5pt}{5pt}

\DeclareMathAlphabet\mathbfcal{OMS}{cmsy}{b}{n}
\makeatletter
\g@addto@macro\normalsize{
\setlength\abovedisplayskip{3pt}
\setlength\belowdisplayskip{3pt}
\setlength\abovedisplayshortskip{0pt}
\setlength\belowdisplayshortskip{0pt}}
\makeatother

\makeatletter
\def\secondauthor#1{\gdef\@secondauthor{#1}}
\renewcommand\maketitle{
    \begin{center}\sffamily
        {\huge\bfseries\@title}\\
            \vspace{6mm}
        {\Large\@author}\\
            \vspace{6mm}
        {\large\@date}
    \end{center}
}
\makeatother

\title{Existence of Lebesgue Measurable Functions Outside the Mauldin Hierarchy}
\author{Senan Sekhon}
\date{\DTMToday}

\begin{document}

\thispagestyle{plain}

\maketitle

\vspace{5pt}

\begin{abstract}
    In 1916, Hausdorff proved that the Baire hierarchy on $\R$, starting with the continuous functions, generates all Borel functions on $\R$. It remained open whether, starting with the a.e. continuous functions, the corresponding hierarchy generates all Lebesgue measurable functions on $\R$. We prove that, assuming the Axiom of Choice, the answer is negative.
\end{abstract}

\vspace{5pt}

\tableofcontents

\vspace{5pt}

\section{Preliminary Definitions}

\begin{itemize}
    \item A \term{$G_\delta$ set} is a countable intersection of open sets, i.e. a set of the form $\cappn U_n$ where each $U_n$ is open. We may assume without loss of generality that this is a countable \emph{decreasing} intersection, i.e. $U_n\sups U_{n+1}$ for all $n\in\N$.
    \item An \term{$F_\sigma$ set} is a countable union of closed sets, i.e. a set of the form $\cuppn F_n$ where each $F_n$ is closed. We may assume without loss of generality that this is a countable \emph{increasing} union, i.e. $F_n\sub F_{n+1}$ for all $n\in\N$.
\end{itemize}

\begin{proposition}\label{set_of_discontinuities_is_F_sigma}
    The set of discontinuities of any function $f:\R\to\R$ is an $F_\sigma$ set.
\end{proposition}

See \cite[Theorem 7.1, page 31]{oxtoby} or \cite[Theorem 6.28, pages 382--384]{thomson} for a proof.

\begin{definition}
    A set $A\sub\R$ has the \term{Baire property} if there is an open set $U\sub\R$ and a meager set $M\sub\R$ such that $A=U\triangle M$.
\end{definition}

\begin{definition}\leavevmode
    \begin{itemize}
        \item A function is of \term{Baire class 0} if it is continuous everywhere.
        \item For each ordinal $\alpha>0$, a function is of \term{Baire class $\alpha$} if it is a pointwise limit of functions of Baire class $<\alpha$.
    \end{itemize}
    We will denote the set of Baire class $\alpha$ functions by $\Bf_\alpha$.
\end{definition}

The collections of functions $\Bf_\alpha$ form the \term{Baire hierarchy} on $\R$. It is well-known that:
\begin{itemize}[align=left, leftmargin=1in, font=\sffamily\itshape]
    \item[The Baire hierarchy does not collapse before $\omega_1$:] $\Bf_\alpha\subset\Bf_{\alpha+1}$ for every ordinal $\alpha<\omega_1$.
    \item[The Baire hierarchy generates all Borel functions:] $\Bf_{\omega_1}=\operatorname{Borel}(\R)$.
\end{itemize}
See \cite[Theorem 24.3, pages 190--191]{kechris} for a proof.

\begin{definition}
    Two functions $f,g:\R\to\R$ are \term{Mauldin equivalent} if there is a null $F_\sigma$ set $F\sub\R$ such that $f=g$ on $\R\exc F$.
\end{definition}
\begin{remarks}\leavevmode
    \begin{enumerate}
        \item Mauldin equivalence is an equivalence relation.
        \item Mauldin equivalence implies equality a.e., since every null $F_\sigma$ set is a null set.
    \end{enumerate}
\end{remarks}

\begin{definition}
    A function is \term{continuous almost everywhere} (\emph{continuous a.e.}) if it is continuous outside a null set, i.e. outside a set of Lebesgue measure zero.
\end{definition}

\begin{definition}\leavevmode
    \begin{itemize}
        \item A function is of \term{Mauldin class 0} if it is continuous a.e.
        \item For each ordinal $\alpha>0$, a function is of \term{Mauldin class $\alpha$} if it is a pointwise limit of functions of Mauldin class $<\alpha$.
    \end{itemize}
    We will denote the set of Mauldin class $\alpha$ functions by $\Mf_\alpha$.
\end{definition}

The collections of functions $\Mf_\alpha$ form the \term{Mauldin hierarchy} on $\R$. Mauldin \cite{mauldin1973} proved that this hierarchy does not collapse before $\omega_1$, i.e. $\Mf_\alpha\subset\Mf_{\alpha+1}$ for every ordinal $\alpha<\omega_1$.\\

A natural question to ask is whether the Mauldin hierarchy generates all Lebesgue measurable functions\footnote{This refers to functions under which the pre-image of every Borel set is Lebesgue measurable.} on $\R$. Since $\Ls(\R)$ is closed under pointwise limits, we trivially have $\Mf_{\omega_1}\sub\Ls(\R)$.\\

Mauldin \cite{mauldin1971,mauldin1973} proved that $f\in\Mf_\alpha$ if and only if there is a function $g\in\Bf_\alpha$ and a null $F_\sigma$ set $F\sub\R$ such that $f(x)=g(x)$ for all $x\in\R\exc F$. Stated in our terminology, we have:

\begin{theorem}\label{mauldin_classes_characterization}
    A function $f:\R\to\R$ is of Mauldin class $\alpha$ if and only if it is Mauldin equivalent to a function of Baire class $\alpha$.
\end{theorem}

\begin{lemma}\label{null_F_sigma_implies_meager}
    Every null $F_\sigma$ set is meager.
\end{lemma}
\begin{proof}
    Suppose $N\sub\R$ is a null $F_\sigma$ set. Since it is $F_\sigma$, it can be expressed as a countable union of closed sets $N=\cuppn F_n$. Since $N$ is null, so are all $F_n$, and so they cannot contain any (non-degenerate) interval. Thus every $F_n$ has empty interior, and since it is closed, it is nowhere dense. Thus $N=\cuppn F_n$ is a countable union of nowhere dense sets, and so it is meager.
\end{proof}

\begin{lemma}\label{baire_property_lemma}
    Suppose $Z\sub\R$ such that $1_Z$ is Mauldin equivalent to a Borel function. Then $Z$ has the Baire property.
\end{lemma}
\begin{proof}
    Suppose $g:\R\to\R$ is a Borel function that is Mauldin equivalent to $1_Z$. Then there is a null $F_\sigma$ set $N\sub\R$ such that $1_Z=g$ on $\R\exc N$. Define $B=g^{-1}(\{1\})$. Then $B$ is a Borel set (since $g$ is a Borel function). Since $1_Z=g$ on $\R\exc N$, we have:
    \begin{equation}\label{baire_property_lemma:eq1}
        Z\exc N = B\exc N
    \end{equation}
    We also have:
    \begin{equation}
        Z\triangle B = (Z\exc B) \cup (B\exc Z)
    \end{equation}
    From \eqref{baire_property_lemma:eq1}, we get $Z\exc B\sub N$ and $B\exc Z\sub N$, and so $Z\triangle B\sub N$. Since $N$ is null and $F_\sigma$, by \Cref{null_F_sigma_implies_meager}, it is meager. Thus $Z\triangle B$ is a subset of a meager set, and so it is meager.\\
    
    Since $B$ is a Borel set, it has the Baire property, so there is an open set $V\sub\R$ and a meager set $M\sub\R$ such that $B=V\triangle M$. This yields:
    \begin{equation*}
        Z\triangle V = Z\triangle (M\triangle B) = (Z\triangle B)\triangle M
    \end{equation*}
    This is the symmetric difference of two meager sets, and so it is meager. Thus $Z=V\triangle (Z\triangle V)$ is the symmetric difference of an open set and a meager set, and so it has the Baire property.
\end{proof}

Before we are ready to prove our main result (\Cref{mauldin_hierarchy_does_not_generate_all_lebesgue_measurable_functions}), we need two more results.

\begin{theorem}[Pettis-Baire Theorem]\label{pettis_baire_theorem}
    Suppose $A\sub\R$ is non-meager and has the Baire property. Then the set $A-A=\setbuild{x-y}{x,y\in A}$ contains an open neighborhood of $0$.
\end{theorem}

See \cite[Theorem 9.9, page 61]{kechris} for a proof.

\begin{theorem}[Baire Category Theorem]\label{baire_category_theorem}
    $\R$ is non-meager in itself.
\end{theorem}

See \cite[Theorem 8.4, pages 41--42]{kechris} or \cite[Theorem 1.3, page 2]{oxtoby} for a proof.

\section{Results Using the Axiom of Choice}

It is well-known that the \href{https://math.vanderbilt.edu/schectex/ccc/choice.html}{Axiom of Choice} implies the existence of a \term{Vitali set}, a set $V\sub\R$ containing exactly one representative of each coset of $\Q$ in $\R$ (as a group under addition). In other words, a Vitali set\footnote{More specifically, a Vitali set \emph{not} containing $0$, otherwise it would not be linearly independent.} is a set of real numbers that forms a basis of $\R$ as a vector space over $\Q$. As shown in \cite[Page 22]{oxtoby}, a Vitali set \emph{cannot} be Lebesgue measurable.\\

We now use this fact to prove the existence of a null (and thus Lebesgue measurable) set that does \emph{not} have the Baire property.

\begin{theorem}\label{null_set_without_baire_property}
    There is a null set without the Baire property.
\end{theorem}
\begin{proof}
    Suppose $(q_k)_{k=1}^\infty$ is an enumeration of $\Q$. For each $n\in\N$, define:
    \begin{equation}
        U_n = \cuppk \qty(q_k-\frac{1}{2^{n+k}},q_k+\frac{1}{2^{n+k}})
    \end{equation}
    Then for all $n\in\N$, $U_n$ is open and dense (since it contains every rational number) and $\mu(U_n)\le\sumk\frac{1}{2^{n+k}}=2^{1-n}$. Define:
    \begin{equation}
        G = \cappn U_n
    \end{equation}
    Then $G$ is $G_\delta$ set. It is also null (since $\limn 2^{1-n}=0$) and comeager (since $\R\exc G=\cuppn (\R\exc G_n)$ is a countable union of closed, nowhere dense sets).\\

    Suppose $V\sub\R$ is a Vitali set, and define $A=V\cap G$. Then $A$ is a null set (since $A\sub G$ and $G$ is null). Suppose $A$ has the Baire property.
    \begin{itemize}
        \item Suppose $A$ is meager. Since $G$ is comeager, $V\exc G$ is meager. Thus $V=(V\cap G)\cup(V\exc G)$ is a union of two meager sets, and so it is meager. This yields:
        \begin{equation}
            \cupp_{q\in\Q} (V+q) = \R
        \end{equation}
        This expresses $\R$ as a countable union of meager sets, contradicting the \hyperref[baire_category_theorem]{Baire category theorem}.
        \item Suppose $A$ is non-meager. By the \hyperref[pettis_baire_theorem]{Pettis-Baire Theorem}, $A-A$ contains an open neighborhood of $0$, and so it contains a nonzero rational number. Thus $V-V$ contains a nonzero rational number, i.e. there exist $v,w\in V$ such that $v-w\in\Q\exc\{0\}$. This contradicts the definition of a Vitali set, which contains exactly one representative of each coset of $\Q$ in $\R$.
    \end{itemize}
    Either way, we get a contradiction. Thus $A$ does not have the Baire property.
\end{proof}

\begin{theorem}\label{mauldin_hierarchy_does_not_generate_all_lebesgue_measurable_functions}
    There is a Lebesgue measurable function $f:\R\to\R$ that is not in $\Mf_{\omega_1}$.
\end{theorem}
\begin{proof}
    By \Cref{null_set_without_baire_property}, there is a null set $Z\sub\R$ without the Baire property. Define $f=1_Z$. Since $Z$ is a null set, it is Lebesgue measurable, and so $f$ is a Lebesgue mesurable function. By \Cref{baire_property_lemma}, $f$ cannot be Mauldin equivalent to \emph{any} Borel function. Since every function in the Baire hierarchy is Borel, $f$ cannot be Mauldin equivalent to any function of any Baire class. By \Cref{mauldin_classes_characterization}, $f$ cannot be of any Mauldin class, and so $f\notin\Mf_{\omega_1}$.
\end{proof}

\section{Extension of the Riemann Integral}

We now present an interesting application to extensions of the Riemann integral. Note that the Riemann integral is equivalent to the Darboux integral. See \cite[Theorem 7.4.11, page 232]{bartle} for a proof.\\

The relation between the Riemann integral and continuity a.e. was first demonstrated by Lebesgue in 1901:

\begin{theorem}[Riemann-Lebesgue Theorem]\label{riemann_lebesgue_theorem}
    A function $f:[a,b]\to\R$ is Riemann integrable if and only if it is bounded and continuous almost everywhere.
\end{theorem}
\begin{remark}
    By \Cref{set_of_discontinuities_is_F_sigma}, it follows that the set of discontinuities of every Riemann integrable function is a null $F_\sigma$ set.
\end{remark}

See \cite[Appendix C, pages 362--363]{bartle}, \cite[Theorems 7.4.7 and 7.4.11, pages 233--236 and 238--239]{botelho} or \cite[Theorem 8.17, pages 518--520]{thomson} for a proof.

\begin{theorem}[Arzelà's Bounded Convergence Theorem]\label{arzela_bounded_convergence_theorem}
    Suppose $(f_n)$ is a uniformly bounded sequence of Riemann integrable functions $f_n:[a,b]\to\R$ that converges pointwise to a Riemann integrable function $f:[a,b]\to\R$. Then:
    \begin{equation*}
        \limn\int_a^b f_n(x)\dx = \int_a^b f(x)\dx
    \end{equation*}
\end{theorem}

See \cite{silva} for a proof.

Note that we need to \emph{assume} the pointwise limit $f$ is Riemann integrable, the standard counterexample being $f_n=1_{\{q_1,q_2,\ldots,q_n\}}$ and $f=1_\Q$, where $(q_n)$ is an enumeration of $\Q$.\\

What we \emph{can} do, however, is extend the notion of the Riemann integral to include more functions than just the Riemann integrable ones. We can do this by \emph{defining} $\int_a^b f(x)\dx$ to be $\limn\int_a^b f_n(x)\dx$. This is well-defined: if $(f_n)$ and $(g_n)$ are two uniformly bounded sequences that converge pointwise to $f$, applying \hyperref[arzela_bounded_convergence_theorem]{Arzelà's Bounded Convergence Theorem} to $(f_n-g_n)$ shows that $\limn\int_a^b f_n(x)\dx=\limn\int_a^b g_n(x)\dx$.\\

For any function $f:[a,b]\to\R$, $\abs{f}$ is continuous wherever $f$ is continuous. Thus $\Mf_0$ is closed under taking absolute values. It is also a vector space, and so it is a lattice (since $\min\{f,g\}=\frac{f+g-\abs{f+g}}{2}$ and $\max\{f,g\}=\frac{f+g+\abs{f+g}}{2}$). Thus for any $f\in\Mf_0$ and any $M>0$, the \emph{truncation} $\min\{\max\{f,-M\},M\}$ is also in $\Mf_0$. Since minima and maxima are preserved by pointwise convergence, this also holds for every Mauldin class $\Mf_\alpha$.\\

We therefore have the following result:

\begin{proposition}
    Suppose $f:[a,b]\to\R$ is bounded and $(f_n)$ is a sequence of functions $f_n:[a,b]\to\R$ such that $f_n\to f$ pointwise. Also suppose that for each $n\in\N$, $f_n\in\Mf_{\alpha_n}$ for some ordinal $\alpha_n$. Then there is a uniformly bounded sequence $(g_n)$ of functions $g_n:[a,b]\to\R$ such that $g_n\to f$ pointwise, and for each $n\in\N$, $g_n\in\Mf_{\alpha_n}$.
\end{proposition}
\begin{proof}
    Since $f$ is bounded, there exists $M>0$ such that $\abs{f}<M$. For each $n\in\N$, define $g_n=\min\{\max\{f,-M\},M\}$. Then $g_n\in\Mf_{\alpha_n}$ for all $n\in\N$, and $g_n\to f$ pointwise.
\end{proof}

This immediately yields the following corollary:

\begin{corollary}
    The Riemann integral can be extended to all bounded functions $f:[a,b]\to\R$ in all Mauldin classes, by repeatedly taking pointwise limits.
\end{corollary}

One might wish to continue extending the integral this way and hopefully recover the Lebesgue integral, at least for bounded Lebesgue measurable functions. Can we do this?\\

By \Cref{mauldin_hierarchy_does_not_generate_all_lebesgue_measurable_functions}, the Mauldin hierarchy does \emph{not} generate all bounded Lebesgue measurable functions. Thus, assuming the Axiom of Choice, the answer is no:

\begin{corollary}
    The Riemann integral \emph{cannot} be extended to all bounded Lebesgue integrable functions $f:[a,b]\to\R$ purely by repeatedly taking pointwise limits.
\end{corollary}

\printbibliography

\end{document}